\documentclass[reqno, 12pt]{amsart} 

\usepackage{amsmath, amssymb} 
\usepackage{amscd}            
\usepackage{amsxtra}          
\usepackage{upref}            

\usepackage[mathscr]{eucal}

\usepackage{mathtools}
\usepackage{verbatim}

\usepackage[foot]{amsaddr}

\theoremstyle{plain}

\newtheorem{theorem}{Theorem}
\newtheorem{proposition}[theorem]{Proposition}
\newtheorem{lemma}[theorem]{Lemma}

\newtheorem{conjecture}[theorem]{Conjecture}

\theoremstyle{remark}
\newtheorem{remark}[theorem]{Remark}

\numberwithin{equation}{section}
\numberwithin{theorem}{section}
\newcommand{\be}%
  {\protect\setcounter{equation}{\value{subsubsection}}}  
\newcommand{\ee}%
  {\protect\setcounter{subsubsection}{\value{equation}}}

\DeclareMathAlphabet\BOONDOX{U}{rsfso}{m}{n}

\newcommand{\R}{{\mathbb R}}
\newcommand{\C}{{\mathbb C}}
\newcommand{\sel}{\BOONDOX{S}}

\newcommand{\esel}{{\sel}^{\#}}

\newcommand{\aselp}{{\mathfrak A}}
\newcommand{\asel}{{\aselp}^{\#}}
\newcommand{\bselp}{\BOONDOX{L}}

\newcommand{\aq}{{\mathbb A}_{\mathbb Q}}
\newcommand{\ak}{{\mathbb A}_k}

\newcommand{\glnak}{{\rm GL}_n(\ak)}

\newcommand{\hktz}{H(\alpha,T)}

\newcommand{\res}{\text{Re}(s)}
\newcommand{\re}{\text{Re}}

\newcommand{\Arg}{\operatorname{Arg}}

\begin{document}


\title{Beyond the extended Selberg class: $1<d_F< 2$}
\author{R. Balasubramanian$^{1,2,3}$}
\author{Ravi Raghunathan$^1$}

\address{$^1$Department of Mathematics \\ 
         Indian Institute of Technology Bombay\\
         Mumbai, 400076\\ India}
\address{$^2$The Institute of Mathematical Sciences\\
            Chennai, 600113\\India}
\address{$^3$Homi Bhabha National Institute\\
            Mumbai, 400094\\India}
\email{balu@imsc.res.in}               
\email{ravir@math.iitb.ac.in}

\begin{abstract} 
We show that a class of Dirichlet series $\asel$ that is much larger than the
extended Selberg class $\esel$, and also contains the standard as well as the tensor product, 
exterior square and symmetric square $L$-functions of automorphic $L$-functions of $GL_n$ 
over number fields, does not have any elements of
degrees between $1$ and $2$.  The proof of our more general theorem is very different
from the proof of Kaczorowski and Perelli for the class $\esel$, and 
is much shorter and simpler even in that case. 
\end{abstract}
\vskip 0.2cm

\maketitle


\markboth{R. BALASUBRAMANIAN AND RAVI RAGHUNATHAN}{On the classification of elements in $\asel$: $1<d<2$}


\section{Introduction}\label{introduction}
Let $s=\sigma+it\in \C$ and let $F(s)$ be a non-zero meromorphic function on $\C$. 
We consider the following conditions on $F(s)$.
\begin{enumerate}
\item [(P1)] The function $F(s)$ is given by a
Dirichlet series $\sum_{n=1}^{\infty}\frac{a_n}{n^{s}}$ 
with abscissa of absolute convergence $\sigma_a\ge 1/2$.
\item[(P2)] There is a polynomial $P(s)$ such that $P(s)F(s)$ 
extends to an entire function, and such that in  any
vertical strip $\sigma_1\le \res \le \sigma_2$, $P(s)F(s)\ll (1+t)^M$ for some $M\in\R$.
\item[(P3)] There exist a real number $Q>0$, a complex number $\omega$ such that 
$\vert\omega\vert=1$, and a function $G(s)$ of the form
\begin{equation}\label{gammafactor}
G(s)=\prod_{j=1}^{r}\Gamma(\lambda_j s+\mu_j)\prod_{j'=1}^{r'}\Gamma(\lambda_{j^{\prime}}^{\prime}s+\mu_{j^{\prime}}^{\prime})^{-1},
\end{equation}
where $\lambda_j, \lambda_{j^{\prime}}^{\prime}>0$, $\mu_j, \mu_{j^{\prime}}^{\prime}\in \C$, 
and $\Gamma(s)$ denotes the usual gamma function, such that
\begin{equation}\label{fnaleqn}
\Phi(s):=Q^{s}G(s)F(s)=\omega\overline{\Phi(1-\bar{s})}.
\end{equation}
\end{enumerate}

We will denote by $\asel$ the set of meromorphic functions satsifying (P1)-(P3). We set 
$d_F=2\sum_{j=1}^{r}\lambda_j-2\sum_{j'=1}^{r^{\prime}}\lambda_{j^{\prime}}^{\prime}$
In \cite{Ragh20} (see Theorem \ref{degwelldef} in Section \ref{degconj}), 
where the class $\asel$ was introduced, we showed that $d_F$ is an invariant
of $F(s)$, that is, it does not depend on the choice of gamma factors $G(s)$ 
in \eqref{fnaleqn}. The subset of $\asel$ consisting of elements of degree 
$d$ will be denoted $\asel_d$ and the subset of elements of degree $d$ with
abscissa of absolute convergence $\sigma_a\le \nu$  will be denoted $\asel_d(\nu)$.
 The main theorem of this paper is
\begin{theorem}\label{mainthm} For $1<d<2$, we have $\asel_d=\emptyset$.
\end{theorem}
In \cite{KaPe11}, Kaczorowski and Perelli proved the result above for the
extended Selberg class $\esel$, that is, when $\nu\le1$, $r'=0$, $\re(\mu_j)\ge 0$ for 
all $1\le j\le r$, and $P(s)=(s-1)^m$ for some integer $m\ge 0$.
This is perhaps the most difficult and interesting of their many
foundational theorems about $\esel$. We emphasise that our proof not only 
works for the (considerably) larger class $\asel$ and uses very different techniques, but also gives a much 
simpler and shorter argument even for the class $\esel$.

As the second author has explained in \cite{Ragh20}, the primary motivation 
for working in $\asel$ rather than in $\esel$ is that the former contains a much larger number 
of series that arise naturally from automorphic forms.
Dispensing with the requirement of absolute convergence for $\res>1$ allows
us to incorporate many more series within our theory.
The class $\asel(1)$ is already known
to contain the standard $L$-functions as well as the exterior square,
symmetric square and tensor product $L$-functions associated
to {\it globally generic} unitary automorphic representations
of $\glnak$, where $k$ is a number field and $\ak$ its
ring of ad\`eles. The class $\asel$ further contains the 
$L$-functions associated to even non-generic automorphic representation of $\glnak$.
The simplest such example is $\zeta(s-1/2)\zeta(s+1/2)$, which is the 
standard $L$-function associated to the trivial representation
of $GL_2(\aq)$. The $L$-series associated to Siegel modular forms
give another class of examples which often do not belong 
to $\esel$ but are part of the class $\asel$.

For the standard $L$-functions of unitary cuspidal automorphic 
$L$-functions, establishing the condition $\re(\mu_j)\ge 0$, $1\le j\le r$,
involves the Generalised Ramanujan Conjecture at infinity. This has not yet been
proven even in the case of the $L$-functions of cuspidal Maass eigenforms,
where the condition is equivalent to the famous Eigenvalue Conjecture
of Selberg. Thus, many of the most important $L$-functions in number theory
are easily seen to lie in $\asel$ or even in $\asel(1)$,
but are not known to lie in $\esel$.
                  
We remark that A. Booker has also defined a class of Dirichlet series $\bselp$ 
in \cite{Booker2015} similar to $\esel$ which contains 
the standard $L$-functions of unitary cuspidal automorphic representations 
of $\glnak$. His requirement of average Ramanujan bounds for the 
mean-square of the coefficients and an Euler product like axiom
makes his class more restrictive than ours. For instance,
the tensor product, exterior square and symmetric square
$L$-functions of even unitary cuspidal automorphic representations of 
$\glnak$ are not known to lie in $\bselp$.  Moreover, $L$-functions
like $\zeta(s-1/2)\zeta(s+1/2)$ which are associated to non-generic representations
do not belong to $\bselp$. 
All linear combinations of series in $\asel$ with the 
same gamma factors and the same factors $\omega Q^s$ in the 
functional equation belong $\asel$ but not, in general, to $\bselp$, since elements of 
the former class need not have an Euler product. Our class $\asel$ thus contains many
examples of series which do not lie in $\bselp$ even conjecturally, contains all the series that
are known to belong to $\bselp$, and almost 
certainly contains all of $\bselp$ (which a priori allows somewhat more general
gamma factors). Theorem \ref{mainthm} may thus be seen as a generalisation and extension of 
Theorem 1.7 of \cite{Booker2015}, which asserts that $\bselp_d=\emptyset$ for $1<d<5/3$.

Allowing Dirichlet series to have a finite number of poles at 
arbitrary locations also results in a larger class, since even the standard
$L$-functions of $GL_n$ may have poles at points other than $1$.
The removal of the condition $r^{\prime}=0$ almost certainly does
not enlarge our class further but does enable us to treat 
quotients of (automorphic) $L$-functions (these last two
weaker hypotheses are also features of $\bselp$).
For further motivation for introducing the class $\asel$ we refer 
to \cite{Ragh20}, where the contrasts with $\esel$ and $\bselp$ 
and some applications are discussed in greater detail.

\begin{remark} As we have also remarked in \cite{Ragh20}, we could allow an even 
more general functional equation of the form $\Phi(s)=Q^{s}\Psi(1-s)$ where $\Psi(s)=\tilde{G}(s)B(s)$ 
for some Dirichlet series $B(s)=\sum_{n=1}^{\infty}b_nn^{-s}$, convergent in some (right) half plane.
We avoid doing this since our paper is already overburdened by notation, even though 
this relaxation actually yields a slightly larger class of series.
\end{remark}

\subsection{A brief sketch of the proof of Theorem \ref{mainthm}} We give an outline of 
the proof of Theorem \ref{mainthm}. Since $\sigma_a$ is the abscissa of absolute convergence of $F(s)$, we have the
estimate $(\sum_{T<n<4T}|a_n|^2)^{1/2}=\Omega (T^{\sigma_a-1/2-\varepsilon})$ for every 
$\varepsilon>0$, where $a_n$ are the coefficients of the Dirichlet series $F(s)$. 
The aim is to get an $O$-estimate for this $\ell^2$-norm incompatible with the 
$\Omega$-estimate above. This is done by using the integral transform (for $T>0$ large enough that
$F(1/2+it)$ is holomorphic when $t>T$)
\[
\hktz :=\frac{1}{\sqrt{\alpha}}
\int_{4\pi\alpha T}^{6\pi\alpha T}
F(1/2+it)e^{it\log\left(\frac{t}{2\pi e\alpha}\right)-i\frac{\pi}{4}}dt,
\]
first introduced in \cite{Sound05} to treat the case $d=1$. A standard argument
allows one to replace $F(1/2+it)$ by an everywhere convergent Dirichlet series and 
error terms $r_1$ and $r_2$, the former arising from the poles of $F(1/2+it+w)$, and the latter
from the inverse Mellin transform of the function $F(1/2+it+w)\Gamma(w/p)$ on a line in 
the half-plane $\sigma<0$.  Applying the integral transform now yields 
\[
\hktz=\sqrt{2\pi}\sum_{T<n<4T}a_ne^{-2\pi i\alpha n}e^{(-n/T^{1+\rho_1})^p}+E+R_1+R_2,
\]
where $E$ is an error term arising from an application of the stationary phase method,
and $R_i$ arises from $r_i$, $i=1,2$ by applying the transform above. 
An important new input is an improvement in the method of stationary phase 
for oscillatory integrals due to Bombieri and Bourgain (Lemma 3.3 of \cite{BoBo04}), 
which is elementary but powerful, and allows the much better estimate 
$E=O(T^{\sigma_a-1/2-\varepsilon_1})$,
for some $\varepsilon_1>0$. This gives a saving of $T$ over the trivial estimate and 
a saving of $T^{2/5}$ over previous methods.
The real difficulty, however, is to obtain a similar saving for $R_2$. 
Our key idea is to use Lemma 3.3 of \cite{BoBo04} for estimating this error term.
This involves using the functional equation and then interchanging the sum and the 
integral by passing to the domain of absolute convergence of $F(s)$. 
This yields the same bound for $R_2$ as the one obtained for $E$. Thus, we get
\[
\hktz=\sqrt{2\pi}\sum_{T<n<4T}a_ne^{-2\pi i\alpha n}e^{(-n/T^{1+\rho_1})^p}+O(T^{\sigma_a-\frac{1}{2}-\varepsilon_1}),
\]
for some $\varepsilon_1>0$, when $\alpha$ and $T$ are chosen sufficiently large (see 
equation \ref{directhk}). On the other hand, applying the functional equation to $F(s)$ first
and then following the procedure above, gives $\hktz=O(T^{\sigma_a-1/2-\varepsilon_1})$ 
(as one sees in equation \eqref{secondhkeqn}).
These two estimates for $\hktz$ together yield $O(T^{\sigma_a-1/2-\varepsilon_1})$ for our
$\ell^2$-norm, for some $\varepsilon_1>0$, contradicting the $\Omega$-estimate above.

With the ideas above, we are able to avoid the delicate and technical analysis of the Fox
hypergeometric functions and other complications that arise in 
\cite{KaPe02} and \cite{KaPe11} altogether, and obtain a more direct proof 
of our more general theorem. Soundararajan has remarked in \cite{Sound05} that his methods 
would work for $\esel_d$ when $d<5/3$, but we have actually been able to extend them to $1<d<2$, 
and for the much larger class of functions $\asel$, because of our 
more refined handling of the error terms.

\section{Preliminaries}\label{prelim}
In the rest of this paper, for any function $a:\C\to \C$, we will use the notation $\tilde{a}(s):=\overline{a(\bar{s})}$.
We will also drop the subscript $F$ and simply use the letter $d$ 
for the degree $d_F$, $\lambda$ for $\lambda_F$ and $\mu$ instead of $\mu_F$. 
\subsection{The degree conjecture}\label{degconj}
In this subsection, we present a few basic results about $\asel$ from \cite{Ragh20}.
In the introduction we mentioned that the degree of an element in 
the class $\asel$ is well defined. The precise result (Theorem 4.1 of \cite{Ragh20}) is the following:
\begin{theorem}\label{degwelldef}  Suppose that $F(s)$ in $\asel$ satisfies 
functional equations $\Phi_i(s)=L_i^sG_i(s)F(s)=\omega_i\tilde{\Phi}_i(1-s)$, 
$i=1,2$, for two different functions $G_i(s)$ having the form given in 
\eqref{gammafactor}. Let $d_i$ denote the degree of $G_i(s)$ for $i=1,2$. Then $d_1=d_2$.
\end{theorem}
As we said in the introduction, the theorem above justifies the notation $d_F$ for 
the degree of an element $F(s)$ of $\asel$, and also the notation $\asel_d$ for the set of all 
$F(s)$ in $\asel$ such that $d_F=d$. We have the following classification theorems
which generalise the results of \cite{CoGh93} and \cite{KaPe99}.
\begin{theorem}\label{degzero} The set $\asel_0$ consists of 
Dirichlet polynomials of the form $\sum_{n\,|\,q}\frac{a_n}{n^s}$ for some integer $q$.
\end{theorem}
In fact, Theorem 4.3 of \cite{Ragh20} shows that $\asel_0=\esel_0$. Thus, the more precise classification 
for series in $\esel_0$ due to Kaczorowski and Perelli in Theorem 1 of \cite{KaPe99} is valid for $\asel_0$. 
We also have Theorem 4.7 of \cite{Ragh20} which generalises the theorems of \cite{CoGh93} and
\cite{Richert57}:
\begin{theorem}\label{degzeroone} For $0<d<1$, $\asel_d=\emptyset$.
\end{theorem}
When $d=1$, $\asel_1(1)$ is non-empty and we have a satisfactory classification of these
series in Theorem 5.1 of \cite{Ragh20} which we state below.
\begin{theorem}\label{degone} If $F(s)$ lies in $\asel_1(1)$ and converges absolutely 
for $\re(s)>1$, there exists a real number $a$ such that $a_nn^{ia}$ is periodic 
with period $q$ for some integer $q$. 
\end{theorem}
It is not hard to see that $F(s)$ must be a linear combination with coefficients in the ring of Dirichlet polynomials 
of the $L$-series $L(s,\chi)$, where the $\chi$ are primitive Dirichlet characters of the modulus $q$.
This theorem was proved in \cite{Ragh20}, by modifying the arguments in \cite{Sound05}.
For the class $\esel$ the analogous result was proved in Theorem 2 of  \cite{KaPe99}. We expect that we can
remove the condition $\sigma_a\le 1$ in the theorem above.
In light of the results above and the main theorem of this paper, it is natural to extend 
the degree conjecture for $\esel$ 
(see Section 1 of \cite{KaPe99}) to our larger class $\asel$.
\begin{conjecture} If $d\ge 0$ is not an integer, then $\asel_d=\emptyset$.
\end{conjecture}

\subsection{Variants of Stirling's formula}
In the domain $\C\setminus (-\infty,0]$ we choose the branch of the logarithm which is real on the positive real axis.
For $z$ lying in a sector of the form $-\pi+\theta_0<\Arg(z)<\pi-\theta_0$ with $\theta_0>0$, 
Stirling's formula states that
\[
\log \Gamma(z)=(z-1/2)\log z-z+\frac{1}{2}\log2\pi+O(1/|z|).
\]
If $z=x+iy$ and $c>0$, then for all $y$ with $|y|>c$, we have
\begin{flalign}
\log \Gamma(z)=(x-1/2)\log|z|-\frac{\pi y}{2|y|}y&+i\left(y\log|z|+(x-1/2)\frac{\pi y}{2|y|}-y\right)\nonumber\\
&+\frac{1}{2}\log 2\pi+O\left(\frac{x^2+1}{|y|}\right).
\end{flalign}
We apply this formula to the quotients of 
Gamma functions that appear in \eqref{fnaleqn} when $s=x+it$. For $1\le k=j\le r$, this yields
\begin{flalign}\label{thirdgammaest}
\frac{\Gamma\left(\lambda_k(1-x-it) +\bar{\mu}_k\right)}{\Gamma\left(\lambda_k(x+it) +\mu_k\right)}
=& \left(\frac{\lambda_kt}{e}\right)^{\lambda_k(1-2x)}e^{-2\lambda_kit\log\frac{t}{e}-2\lambda_kit\log\lambda_k+(\bar{\mu}_k-\mu_k)\log t}
\nonumber\\
&\times e^{(\bar{\mu}_k-\mu_k)\log\lambda_k-(\mu_k+\bar{\mu}_k+\lambda_k)i\frac{\pi}{2}+(\mu_k-\bar{\mu}_k)}
\cdot(1+O(1/t)),
\end{flalign}
where the implied constant depends  on $x$, $\lambda_k$ and $\mu_k$, and $t$ is large enough.
A similar formula holds when $\lambda_k$ and $\mu_k$ are replaced by 
$\lambda_{k^{\prime}}^{\prime}$ and $\mu_{k^{\prime}}^{\prime}$ respectively.
Taking the product over all $j$ and $j'$, we get
\begin{equation}\label{fourthgammaest}
\frac{\tilde{G}(1-x-it)}{G(x+it)}
= (Ce^{-d}t^{d})^{(\frac{1}{2}-x)} e^{-itd\log\frac{t}{e}}t^{iA}e^{iB}C^{-it}\cdot 
(1+O(1/t)),
\end{equation}
where 
\[
A=-i((\bar{\mu}-\mu)-(\bar{\mu^{\prime}}-\mu^{\prime})),\quad C=\prod_{j,j'=1}^{r,r'}{\lambda_j}^{2\lambda_j}
{\lambda_{j^{\prime}}^{\prime}}^{-2\lambda_{j^{\prime}}^{\prime}}, 
\]
and
\begin{flalign}
B=&-i\left(\sum_{j=1}^r(\bar{\mu}_j-\mu_j)\log\lambda_j-\sum_{j'=1}^r(\overline{\mu_{j^{\prime}}^{\prime}}-\mu_{j^{\prime}}^{\prime})\log\lambda_{j,}\right)\nonumber\\
&-(\mu-\bar{\mu})+(\mu^{\prime}-\bar{\mu^{\prime}})-((\mu-\bar{\mu})-(\mu^{\prime}-\bar{\mu^{\prime}})+d/2)\frac{\pi}{2},
\end{flalign}
with
\[
\mu=\sum_{j=1}^{r}\mu_j\quad\text{and}\quad \mu^{\prime}=\sum_{j^{\prime}=1}^{r}\mu_{j^{\prime}}^{\prime}.
\]
Note that $A\in \R$ and $C>0$. Replacing $x+it$ by $x+it+w$ ($w=u+iv$) in \eqref{fourthgammaest}, 
and taking absolute values, we obtain
\begin{equation}\label{zerothgammaest}
\frac{\tilde{G}(1-x-it-w)}{G(x+it+w)}\ll (1+|t+v|)^{-d(x-1/2+u)}.
\end{equation}

\subsection{Replacing $F(z+it)$ by a convergent Dirichlet series}
In this subsection we prove a lemma allowing one to represent the
function $F(z+it)$ by the sum of a Dirichlet series convergent at all points
$s=z+it$ where $F(s)$ is holomorphic, and error
terms over which one has considerable control. A proof of a slightly different
version of this lemma has appeared in \cite{Ragh20}, but since the changes
we have made are crucial, we give a proof in this paper as well. 
\begin{lemma}\label{basiclemma} Let $w=u+iv$, $z=x+iy$, $p>0$ and $d>0$.
If $F(s)$ is holomorphic at $s=z+it$ and $0<\eta<1-x+p-\sigma_a$, we have
\begin{equation}\label{basiclemmaeqn}
F(z+it)=\sum_{n=1}^{\infty}\frac{a_ne^{-(n/X)^{p}}}{n^{z+it}}
+r_1(t)+r_2(t),
\end{equation}
where $r_1(t):=r_1(z,t,X)=O(X^{\sigma_a-x}e^{-|t|/p})$ 
is identically zero if $F(z)$ is entire, and 
\begin{equation}\label{rtwoz}
r_2(t):=\frac{1}{2\pi i p}\int_{u=-p+\eta}F(z+it+w)X^{w}\Gamma(w/p) dw
\ll O(t^{d(\frac{1}{2}+p-x-\eta)}X^{-p+\eta}),
\end{equation}
where $u=-p+\eta$ is a line on which none of the poles of $F(z)$ lie. 
\end{lemma}

\begin{proof}
When $c>\sigma_a$, we have 
\[
\frac{1}{2\pi ip}\int_{u=c}F(z+it+w)X^w\Gamma(w/p)dw
=\sum_{n=1}^{\infty}\frac{a_ne^{-(n/X)^{p}}}{n^{z+it}}.
\]
If $\beta$ is a pole of $F(z)$,  when we move the line of 
integration from $u=c$ to $u=-p+\eta$ we will cross the poles of the 
integrand of the form 
$w=\beta-z-it$, and also pole at $w=0$ (this shifting of the line is made possible by the Phragm\'en-Lindel\"of
principle, since (P2) shows that $F(z)$ has at most polynomial growth in lacunary vertical strips).
The residue at $w=0$ is $F(z+it)$. For any pole $\beta$ of $F(z)$, 
we must have $\re(\beta)\le \sigma_a$, since $\sigma_a$ is the abscissa of 
absolute convergence. Hence, the residue at $\beta-z-it$ will be majorised by 
$X^{\sigma_a-x}e^{-|t|/p}$. We denote the sum of the residues by $r_1(t)$. It is 
obviously an entire function which is identically zero if $F(z)$ is entire and is
bounded by $X^{\sigma_a-x}e^{-|t|/p}$ in any case. It remains to estimate
\[
r_2(t)=\frac{1}{2\pi ip}\int_{u=-p+\eta}F(z+it+w)X^{w}\Gamma(w/p) dw.
\]
Using the functional equation \eqref{fnaleqn}, we obtain the expression
\[
\frac{1}{2\pi ip}\int_{u=-p+\eta}\omega Q^{1-2x-2it-2w}\frac{\tilde{G}(1-z-it-w)}{G(z+it+w)}\tilde{F}(1-z-it-w)
X^{w}\Gamma(w/p) dw
\]
for $r_2(t)$. Since $0<\eta<1-x+p-\sigma_a$, the series $\tilde{F}(1-z-it-w)$ converges absolutely and is bounded as
$t\to \infty$. The bound for $r_2(t)$ now follows easily from \eqref{zerothgammaest}.
This completes the proof.
\end{proof}
\begin{remark}\label{lemremark} One has the obvious analogue of the lemma 
above for the function $\tilde{F}(z)$ and it is proved in exactly the same way. 
\end{remark}
\begin{remark} What we will really need is that $p$ is large compared to $\sigma_a$. 
\end{remark}

\subsection{Estimates for Oscillatory integrals}
We will require two lemmas for estimating and evaluating oscillatory integrals.
The first lemma (which is standard) can be found in Section 1.2 of Chapter VIII in \cite{Stein93}.
\begin{lemma}\label{largen} Suppose that $g(t)$ is a function of bounded variation on 
an interval $K=[a,b]$ and $|g(t)|\le M$ for all $t\in K$. For any ${\mathcal C}^1$-function $f$ on $K$,
if $f^{\prime}(t)$ is monotonic and $|f^{\prime}(t)|\ge m_1$ on $K$,
\[
\int_K g(t)e^{if(t)}dt \ll \frac{1}{m_1}\left\{\vert M\vert +\int_K\vert g^{\prime}(t)\vert dt\right\}.
\]
\end{lemma}
We will also use a modified version of Lemma 3.3 of \cite{BoBo04}. This improves the
usual estimates given by stationary phase techniques and yields much sharper error terms.
\begin{lemma}\label{midn}
Suppose that $f$ is a ${\mathscr C}^3$-function on an interval 
$K=[a,b]$, $f^{\prime\prime}(t)\ne 0$ on $K$. Let $g$ be a smooth bounded 
function on $K$ satsifying $|g(t)|\le M$, $g^{\prime}(t)\ll M/(b-a)$ and 
$g^{\prime\prime}(t)\ll M/\log (b-a)$. If $f^{\prime}(c)=0$ for some $c\in K$ and $m>0$ is such 
that  $|f^{\prime\prime\prime}(t)|\le m$ for 
$t\in K\cap [c-\left\vert\frac{f^{\prime\prime}(c)}{m}\right\vert,c+\left\vert\frac{f^{\prime\prime}(c)}{m}\right\vert]$,
then
\[
\int_K g(t)e^{if(t)}dt=e^{\pm i\frac{\pi}{4}}\frac{g(c)e^{if(c)}}{\sqrt{|f^{\prime\prime}(c)|}}
+O\left(\frac{mM}{|f^{\prime\prime}(c)|^2}\right)
+O\left(\frac{M}{|f^{\prime}(a)|}+\frac{M}{|f^{\prime}(b)|}\right).
\]
The $\pm$ in the expression above occurs according to the sign of $f^{\prime\prime}(c)$.
\end{lemma}
The proof of the lemma above follows  the proof of Lemma 3.3 of \cite{BoBo04}, which corresponds
to the case $g(t)\equiv 1$. Since the arguments for an arbitrary function
$g(t)$ satisfying the conditions of the lemma are virtually identical, we do not repeat them here.

\section{Evaluating $\hktz$ directly}\label{directhktzsec}
Let $d\ge 1$.
Since we are assuming that $F(s)$ is absolutely convergent for $\res>\sigma_a$, we have the estimate
\begin{equation}\label{averageram}
\sum_{n<X}|a_n|\ll_{\varepsilon} X^{\sigma_a+\varepsilon}
\end{equation}
for any $\varepsilon>0$. On the other hand, since we are assuming that the abscissa of absolute convergence is
$\sigma_a$, the Cacuhy-Schwartz inequlaity gives the $\Omega$-estimate
\begin{equation}\label{ltwolowerbound}
\left(\sum_{T<n<4T} |a_n|^2\right)^{\frac{1}{2}}=\Omega(T^{\sigma_a-\frac{1}{2}-\varepsilon})
\end{equation}
for any $\varepsilon>0$.

For any $\alpha\ge 1$ and $T$ large enough that $F(1/2+it)$ is holomorphic for $t\ge 4\pi\alpha T$, we define 
\begin{equation}\label{hkdefn}
\hktz :=\frac{1}{\sqrt{\alpha}}
\int_{K_T}
F(1/2+it)e^{it\log\left(\frac{t}{2\pi e\alpha}\right)-i\frac{\pi}{4}}dt,
\end{equation}
where $K_T=[4\pi\alpha T,6\pi\alpha T]$. We assume that $X_1=T^{1+\rho_1}$ with $\rho_1>0$, and 
that $T$ is chosen large enough that $T^{1+\rho_1}>4T$. A more precise choice of $\rho_1$ will be made
later. If we apply this transform to both sides of Lemma \ref{basiclemma} for $x=1/2$ and $X=X_1$, 
we obtain
\begin{equation}\label{bltransform}
\hktz=\frac{1}{\sqrt{\alpha}}\sum_{n=1}^{\infty}\frac{a_n}{\sqrt{n}}e^{-(n/X_1)^{p}}I_n
+R_1(\alpha,T)+R_2(\alpha,T),
\end{equation}
where $R_i(\alpha,T)=\frac{1}{\sqrt{\alpha}}\int_{K_T}r_i(t)e^{it\log\left(\frac{t}{2\pi e\alpha}\right)-i\frac{\pi}{4}}dt$,
for $i=1,2$, and 
\[
I_n=I_n(\alpha):=\frac{1}{2\pi i} \int_{K_T}e^{it\log\left(\frac{t}{2\pi e\alpha n}\right)
-i\frac{\pi}{4}}dt.
\]
\subsection{Evaluation of the main term} \label{subsecmt}
We evaluate the main term of \eqref{bltransform} involving the integral $I_n$ first, using 
Lemma \ref{largen} and Lemma \ref{midn}. Using the notation of those lemmas, 
we note that $I_n$ has the form $\int_Ke^{if(t)}dt$, where $K=K_T$,
and $f(t)$ is a smooth function on all of $\R^{+}$. Explicitly, we have
\begin{equation}\label{fdefn}
f(t)=t(\log t-\log 2\pi e\alpha n)-\pi/4
\end{equation}
and $g(t)\equiv 1$. We will also need the first three derivatives of 
$f(t)$ for our analysis:
\[
f^{\prime}(t)=\log(t/2\pi\alpha n),
\quad
f^{\prime\prime}(t)=1/t,
\quad
\text{and}\quad
f^{\prime\prime\prime}(t)=-1/t^2.
\]
We note that if $n\le T$ or $n\ge 4T$, 
$f^{\prime}(t)\ne 0$ in $K_T$. We will thus estimate the integrals 
$I_n$ for $n$ in these two ranges first. We treat the case when $n\ge 4T$
using Lemma \ref{largen}. Because of the exponential decay of the term
$e^{-(n/X_1)^{p}}$ when $n$ exceeds $X_1^{1+\varepsilon}$, it is enough to 
sum $n$ between $4T$ and $X_1^{1+\varepsilon}$  in \eqref{bltransform}.
In our situation, we take $g(t)=1$, $f(t)$ as in \eqref{fdefn} and $K=K_T$, 
as above, so when $n\ge 4T$,
\[
|f^{\prime}(t)|\ge 
\left\vert \log(6\pi\alpha T/8\pi \alpha T)\right\vert 
\ge \left\vert\log3/4\right\vert>0.
\]
It follows that $I_n=O(1)$ in this range. For $n\le T$, we have
\[
|f^{\prime}(t)|\ge \left\vert\log(4\pi\alpha T/2\pi\alpha T)\right\vert=\vert\log 2\vert>0.
\]
So, once again, $I_n=O(1)$ in this range. It follows from \eqref{averageram} that 
\begin{equation}\label{largenestprefe}
\frac{1}{\sqrt{\alpha}}\sum_{4T\le n}\frac{a_n}{\sqrt{n}}e^{-(n/X_1)^{p}}\cdot I_n
=O\left(\alpha^{-\frac{1}{2}}X_1^{\sigma_a-\frac{1}{2}+\varepsilon}\right)
\end{equation}
and 
\begin{equation}\label{smallnestprefe}
\frac{1}{\sqrt{\alpha}}\sum_{1 \le n\le T}\frac{a_n}{\sqrt{n}}e^{-(n/X_1)^{p}}\cdot I_n
=O\left(\alpha^{-\frac{1}{2}}X_1^{\sigma_a-\frac{1}{2}+\varepsilon}\right).
\end{equation}

To deal with $T<n<4T$, we require Lemma \ref{midn}. 
As before, we take $g(t)\equiv 1$ and $f(t)$ as in \eqref{fdefn}.
We observe that $f^{\prime}(c)=0$ exactly when
$c=2\pi\alpha n$, so $f^{\prime\prime}(c)=1/c=1/(2\pi\alpha)n$. 
In the interval $[c-\eta,c+\eta]$, $f^{\prime\prime\prime}(t)$ takes its
maximum at $c-\eta$. Thus, we require that 
\[
\frac{1}{\left(c-\frac{1}{cm}\right)^2}\le m,
\]
which is equivalent to the condition
\[
c^2m^2-3m+\frac{1}{c^2}\ge 0.
\]
We can easily verify that $m=3/c^2$ satisfies the above inequality. 
We have $a=6\pi\alpha T$ and $b=8\pi\alpha T$, which yield
\[
O\left(\frac{m}{|f^{\prime\prime}(c)|^2}\right)=O(1),\quad
O\left(\frac{1}{|f^{\prime}(a)|}\right)=O(1)\quad\text{and}\quad 
O\left(\frac{1}{|f^{\prime}(b)|}\right)=O(1)
\]
in both $\alpha$ and $T$ when $T< n< 4T$. Thus, Lemma \ref{midn} yields
\[
I_n=(2\pi\alpha)^{1/2}\sqrt{n}e^{-2\pi i\alpha n}+O(1).
\]

Combining this with \eqref{largenestprefe} and \eqref{smallnestprefe}, we get (for $X_1>4T$)
\begin{equation}\label{midnestpostfe}
\frac{1}{\sqrt{\alpha}}\sum_{n=1}^{\infty}
\frac{a_n}{\sqrt{n}}e^{-(n/X_1)^{p}}\cdot I_n
=\sqrt{2\pi}\sum_{T<n<4T}a_ne^{-2\pi i\alpha n}e^{(-n/X_1)^p}
+O\left(\alpha^{-\frac{1}{2}}X_1^{\sigma_a-\frac{1}{2}+\varepsilon}\right).
\end{equation}

\subsection{Estimating the error terms} \label{forwarderror} 
Recall that $X_1=T^{1+\rho_1}>4T$.
From the estimate for $r_1(t)$ in Lemma \ref{basiclemma} we have
$R_1(\alpha,T)=O(\alpha^{-\frac{1}{2}} X_1^{\sigma_a-1/2}e^{-\alpha T/p})$. 
This term thus decays exponentially in $\alpha$ and $T$. 
Further, if $F(z)$ is entire, we have $R_1(\alpha,T)=0$. 
We may thus disregard the term $R_1(\alpha,T)$ in future since it will be dominated
by the other terms.

Recalling the definition of $R_2(\alpha,T)$ from  equation \eqref{rtwoz},
and choosing $\eta=\eta_1$ and $u_1=-p+\eta_1$ for
$0<\eta_1<p-\sigma_a+1/2$, we see that
\[
R_2(\alpha,T)=\frac{1}{2\pi ip\sqrt{\alpha}}\int_{K_T}\int_{u=u_1}F(1/2+it+w)X_1^{w}\Gamma(w/p) dw
e^{it\log\left(\frac{t}{2\pi e\alpha}\right)-i\frac{\pi}{4}}dt.
\]
Applying the functional equation \eqref{fnaleqn} and interchanging the
order of integration, we obtain
\begin{flalign}
R_2(\alpha,T)=\frac{\omega}{2\pi ip\sqrt{\alpha }}
\int_{u=u_1}\int_{K_T}Q^{-2w-2it}&\frac{\tilde{G}(1/2-it-w)}{G(1/2+it+w)}
\tilde{F}(1/2-it-w)\nonumber\\
&\times e^{it\log\left(\frac{t}{2\pi e\alpha}\right)-i\frac{\pi}{4}}dt
\,\,X_1^{w}\Gamma(w/p) dw.
\end{flalign}
We denote the inner integral by $J(w)$.
We break the outer
integral at $v=\pm\log^2(\alpha T)$ to produce a sum of three integrals:
\begin{flalign}\label{threeint}
&R_2(\alpha,T)=\frac{\omega}{2\pi ip\sqrt{\alpha }}\Big(\int_{u_1-i\infty}^{u_1-i\log^2(\alpha T)}J(w)
X_1^{w}\Gamma(w/p) dw\nonumber\\
&+\int_{u_1+i\log^2(\alpha T)}^{u_1+i\infty}J(w)
X_1^{w}\Gamma(w/p) dw
+\int_{u_1-i\log^2(\alpha T)}^{u_1+i\log^2(\alpha T)}J(w)
X_1^{w}\Gamma(w/p) dw\Big).
\end{flalign}
Using \eqref{zerothgammaest}, it is easy to see that when $|v|\ge \log^2(\alpha T)$,
$J(w)$ can be bounded by $e^{c\log(\alpha T)}$ for some $c>0$. On the other hand,
we know that $\Gamma(w/p)=O(e^{-\log^2(\alpha T/p)})$ in this range. 
It follows that the first and second integrals in \eqref{threeint} are
$O(1)$ (recall that $X_1=T^{1+\rho_1}$). 

We will use the notation
$a(t)\sim b(t)$ to mean $c_1b(t)<a(t)<c_2b(t)$ for some constants $c_1,c_2>0$, and $t$
large enough.
It remains to bound the third integral. In the range of integration, we have $t+v>0$ and
$t+v\sim \alpha T$.
Using \eqref{fourthgammaest}, we can write
\begin{equation}
J(w)=\int_{K_T}(DQ^2(t+v)^d)^{p-\eta_1}e^{if_0(t)}
\tilde{F}(1/2-it-w)(1+O(1/t))dt,\nonumber
\end{equation}
where
\[
f_0(t)=-d(t+v)\log (e^{-1}(CQ^2)^{1/d}(t+v))
+t\log(t/2\pi e\alpha)+A\log (t+v)+B-\frac{\pi}{4}
\]
and $D=Ce^{-d}$.
When $\eta_1<p-\sigma_a+1/2$ the series $\tilde{F}(1/2-it-w)$ converges
absolutely, so we can switch the integral and the sum in the expression for 
$J(w)$. Thus we can rewrite $J(w)$ as 
\begin{equation}\label{innerint}
\sum_{n=1}^{\infty}\frac{\overline{a_n}}{n^{\frac{1}{2}-w}}
(J_n^{(1)}+J_n^{(2)}).
\end{equation}
Here
\begin{equation}
J_n^{(1)}=\int_{K_T}g_1(t)e^{if_1(t)}dt\quad\text{and}\quad 
J_n^{(2)}=\int_{K_T}g_1(t)e^{f_1(t)}O(1/t)dt,\nonumber
\end{equation}
where 
\[
f_1(t)=-d(t+v)\log (e^{-1}(CQ^2)^{1/d}(t+v))
+t\log(tn/2\pi e\alpha)+A\log (t+v)+B-\frac{\pi}{4},
\]
and $g_1(t)=(DQ^2(t+v)^d)^{p-\eta_1}$.
We estimate the integrals $J_n^{(1)}$ and $J_n^{(2)}$ using Lemma \ref{largen}
and Lemma \ref{midn}. It will be useful to set $q=2\pi CQ^2$.
We have
\[
 f_1^{\prime}(t)=-d\log\left((q\alpha n^{-1})^{1/d}t^{-1/d}(t+v)\right)+A/(t+v),
\]
\[
 f_1^{\prime\prime}(t)=-(d-1)/(t+v)+v/t(t+v)-A/(t+v)^2\,\,\text{and}\,\,
\]
\[
 f_1^{\prime\prime\prime}(t)=(d-1)/(t+v)^2+(2t+v)/t^2(t+v)^2+2A/(t+v)^3
\]
Assume now that $d>1$. 
We note that for $t$ large and $v$ much smaller than $t$, the behaviour of $f_1(t)$ 
is very similar to that of $f(t)$ in equation \eqref{fdefn}, and the
integral $J_n^{(1)}$ can thus be estimated in a manner similar to the way we estimated $I_n$.
In particular, we have 
$ f_1^{\prime\prime}(t)\sim -1/t$ and $f_1^{\prime\prime\prime}(t)\sim1/t^2$,
since the second and third terms that occur in the expressions for these derivatives are much smaller
than the first. As before, we need to estimate the sum in three different ranges. We 
denote the set of 
$n$ satisfying $n<q\alpha^{d}(2\pi T)^{d-1}$ by 
$L_{(d-1),T}^{<}$, the set of $n$ in the interval $[q\alpha^{d}(2\pi T)^{d-1},
q\alpha^{d}(8\pi T)^{d-1}]$ by $L_{(d-1),T}$, and the set of 
$n$ satisfying $n>q\alpha^{d}(8\pi T)^{d-1}$ by $L_{(d-1),T}^{>}$.
Note that $f^{\prime}(x_n)=0$ and $x_n \in K_T$, 
means that $n\in L_{(d-1),T}$ for large $\alpha T$. Thus, we rewrite the first term of \eqref{innerint} as
\begin{equation}
\sum_{n=1}^{\infty}\frac{\overline{a_n}}{n^{\frac{1}{2}-w}}J_n^{(1)}=S_1(w)+S_2(w)+S_3(w),\nonumber
\end{equation}
where
\begin{flalign}\label{threeranges}
&S_1(w)=\sum_{n\in L_{(d-1),T}^{<}}\frac{\overline{a_n}}{n^{\frac{1}{2}-w}}
J_n^{(1)},\nonumber\\
&S_2(w)=\sum_{n\in L_{(d-1),T}}\frac{\overline{a_n}}{n^{\frac{1}{2}-w}}
J_n^{(1)}\,\,\text{and}\,\,\nonumber\\
&S_3(w)=\sum_{n\in L_{(d-1),T}^{>}}\frac{\overline{a_n}}{n^{\frac{1}{2}-w}}
J_n^{(1)}.\nonumber
\end{flalign}
We first majorise the contribution of $\int_{u=u_1}S_1(w)X_1^w\Gamma(w/p)dw$ 
to $R_2(\alpha,T)$. We take 
$\eta_1=p-\sigma_a+1/2-\varepsilon$ so the line of integration is
$u_1=-\sigma_a+1/2-\varepsilon$ 
in this case. We use Lemma \ref{largen} to estimate this sum.
In the notation of that lemma, we have $m_1\ge \log 2$ for $\alpha<T$ large enough,
since $ f_1^{\prime}(t)\sim \log (\alpha n^{-1})^{1/d}t^{1-1/d} $ in that case, 
while $M=g_1(b)$ is bounded above by $O((\alpha T)^{d(\sigma_a-\frac{1}{2}+\varepsilon)})$. The term
$\int_K|g_1^{\prime}(t)|dt$ is also bounded by $O((\alpha T)^{d(\sigma_a-\frac{1}{2}+\varepsilon)})$ for 
any $\varepsilon>0$. It follows
from \eqref{averageram} that $S_1(w)=O((\alpha T)^{d(\sigma_a-\frac{1}{2}+\varepsilon)})$. 
Hence,
\begin{equation}\label{soneest}
\frac{\omega}{2\pi ip\sqrt{\alpha}}\int_{u_1-i\log^2\alpha}^{u_1+i\log^2\alpha}
S_1(w)X_1^w\Gamma(w/p)dw=O(\alpha^{-\frac{1}{2}}(\alpha^dT^dX_1^{-1})^{\sigma_a-\frac{1}{2}+\varepsilon}).
\end{equation}
To estimate $S_3(w)$ we take $\eta_1=p-\sigma_a+1/2-\varepsilon$ so the line of integration is $u_1=-\sigma_a+1/2-\varepsilon$. Using Lemma \ref{largen} as above we obtain the 
bound $m_1\ge \log 4/3$ for $\alpha$ large enough. It follows from \eqref{averageram} that
\begin{equation}\label{sthreeest}
\frac{\omega}{2\pi ip\sqrt{\alpha}}\int_{u_1-i\log^2\alpha}^{u_1+i\log^2\alpha}
S_3(w)X_1^w\Gamma(w/p)dw=O(\alpha^{-\frac{1}{2}}(\alpha^dT^dX_1^{-1})^{\sigma_a-\frac{1}{2}+\varepsilon}).
\end{equation}

For estimating $S_2(w)$ we use Lemma \ref{midn}. 
If $f_1^{\prime}(x_n)=0$, we have
\begin{flalign}\label{jnoneeval}
J_n^{(1)}=(DQ^2)^{d(p-\eta_1)}(x_n+v)^{d(p-\eta_1)}(f_1^{\prime\prime}(x_n))^{-\frac{1}{2}}
&a_ne^{if_1(x_n)}+O((\alpha T)^{d(p-\eta_1)})\nonumber\\
&=:M_n+E_n.\nonumber
\end{flalign}
We recall that for $n\in L_{(d-1),T}$ we have $n\sim T^{d-1}\alpha^{d}$ and 
$x_n\sim \alpha T$. So for $n\in L_{(d-1),T}$, we have $f_1^{\prime\prime}(x_n)\sim 1/\alpha T$
and $f_1^{\prime\prime\prime}(x_n)\sim1/(\alpha T)^2$. It follows that
$M_n=O((\alpha T)^{d(p-\eta_1)+\frac{1}{2}})$.

When dealing with the error term $E_n$, we choose 
$\eta_1=p-\sigma_a+1/2-\varepsilon$ as before. Thus, 
$E_n=O((\alpha T)^{d(\sigma_a-1/2+\varepsilon})$. Since we are once again
in the domain of absolute convergence, 
\begin{equation}\label{encont}
\sum_{n\in L_{(d-1),T}}\frac{\overline{a_n}}{n^{\frac{1}{2}-w}}E_n=O((\alpha T)^{d(\sigma_a-1/2+\varepsilon)}).
\end{equation}
When dealing with $M_n$, we choose
$\eta_1=\varepsilon$. On the line $u_1=-p+\varepsilon$, we have
$M_n=O((\alpha T)^{d(p-\varepsilon)+1/2}$. On this line we
have $|n^{\frac{1}{2}-w}|\ge (q\alpha^{d}(2\pi T)^{d-1})^{p+1/2-\varepsilon}$.
Hence, summing over $n$ yields 
\begin{equation}\label{mncont}
\sum_{n\in L_{(d-1),T}}\frac{\overline{a_n}}{n^{\frac{1}{2}-w}}M_n=
O(\alpha^{\frac{1}{2}}(\alpha^dT^{d-1})^{\sigma_a-\frac{1}{2}+\varepsilon}T^{p-\varepsilon+\frac{1}{2}}).
\end{equation}
Using \eqref{encont} and \eqref{mncont}, we obtain
\begin{flalign}\label{stwoest}
\frac{\omega}{2\pi ip\sqrt{\alpha}}\int_{u_1-i\log^2\alpha}^{u_1+i\log^2\alpha}
S_2(w)X_1^w\Gamma(w/p)dw&=O(\alpha^{-\frac{1}{2}}(\alpha^dT^dX_1^{-1})^{\sigma_a-\frac{1}{2}+\varepsilon})\nonumber\\
&+O((\alpha^dT^{d-1})^{\sigma_a-\frac{1}{2}+\varepsilon}T^{p-\varepsilon+\frac{1}{2}}X_1^{\varepsilon-p}).
\end{flalign}
It is trivial to see that $J_n^{(2)}=O(\alpha^{d(p-\eta_1)})$. As in the error
term for $J_n^{(1)}$, we choose $\eta_1=p-\sigma_a+1/2-\varepsilon$. We thus obtain
\begin{equation}\label{jntwocont}
\frac{\omega}{2\pi ip\sqrt{\alpha}}\int_{u_1-i\log^2\alpha}^{u_1+i\log^2\alpha}\sum_{n=1}^{\infty}\frac{\overline{a_n}}{n^{\frac{1}{2}-w}}J_n^{(2)}X_1^w\Gamma(w/p)dw
=O(\alpha^{-\frac{1}{2}}(\alpha^dT^dX_1^{-1})^{\sigma_a-\frac{1}{2}+\varepsilon}).
\end{equation}
Combining the estimates in \eqref{midnestpostfe}, \eqref{soneest}, \eqref{sthreeest}, \eqref{stwoest} and 
\eqref{jntwocont}, we obtain
\begin{flalign}\label{preprelimhkeqn}
&\hktz=\sqrt{2\pi}\sum_{T<n<4T}a_ne^{-i(2\pi\alpha)n}e^{(-n/T^{1+\rho_1})^p}
+O\left(\alpha^{-\frac{1}{2}}X_1^{\sigma_a-\frac{1}{2}+\varepsilon}\right)\nonumber\\
&+O(\alpha^{-\frac{1}{2}}(\alpha^dT^dX_1^{-1})^{\sigma_a-\frac{1}{2}+\varepsilon})
+O((\alpha^dT^{d-1})^{\sigma_a-\frac{1}{2}+\varepsilon}T^{p-\varepsilon+\frac{1}{2}}X_1^{\varepsilon-p})
\end{flalign}
for $X_1=T^{1+\rho_1}>4T$. This yields
\begin{flalign}\label{prelimhkeqn}
&\hktz=\sqrt{2\pi}\sum_{T<n<4T}a_ne^{-i(2\pi\alpha)n}e^{(-n/T^{1+\rho_1})^p}+O(\alpha^{-\frac{1}{2}}T^{(1+\rho_1)(\sigma_a-\frac{1}{2}+\varepsilon)})\nonumber\\
&+O(\alpha^{d(\sigma_a-1/2+\varepsilon)-\frac{1}{2}}T^{(d-1-\rho_1)(\sigma_a-\frac{1}{2}+\varepsilon)})
+O(\alpha^{d(\sigma_a-\frac{1}{2}+\varepsilon)}T^{(d-1)(\sigma_a-\frac{1}{2}+\varepsilon)}T^{\rho_1(\varepsilon-p)+\frac{1}{2}}).
\end{flalign}  
We now assume $\alpha\in [T^{\delta},T^{\delta}+1]$, where $\delta>0$ will be chosen sufficiently small. Then
\begin{flalign}\label{nexthkeqn}
&\hktz=\sqrt{2\pi}\sum_{T<n<4T}a_ne^{-i(2\pi\alpha)n}e^{(-n/T^{1+\rho_1})^p}+O(T^{(1+\rho_1)(\sigma_a-\frac{1}{2}+\varepsilon)-\frac{\delta}{2}})\nonumber\\
&+O(T^{(d-1-\rho_1+d\delta)(\sigma_a-\frac{1}{2}+\varepsilon)-\frac{\delta}{2}})
+O(T^{d(1+\delta)(\sigma_a-\frac{1}{2}+\varepsilon)}T^{\rho_1(\varepsilon-p)+\frac{1}{2}+d\varepsilon}).
\end{flalign}  

Now we assume further that $d<2$ (recall that we are already assuming that $d>1$
and $\sigma_a\ge 1/2$). We choose 
\begin{equation}\label{allchoices}
\delta=\frac{2-d}{\left(d-\frac{1}{4\sigma_a}\right)},\,\,\rho_1=\frac{\delta}{8\sigma_a}\,\,\text{and}\,\,\varepsilon=\frac{\delta}{20}.
\end{equation}
Note that $0<\delta\le 2$ by choice, so $0<\varepsilon=\delta/20\le1/10$. Thus,
$\sigma_a-1/2+\varepsilon<\sigma_a$, so $\rho_1(\sigma_a-1/2+\varepsilon)<\delta/8$.
Hence, 
\[
(1+\rho_1)(\sigma_a-1/2+\varepsilon)-\delta/2<\sigma_a-1/2+\delta/20+\delta/8-\delta/2
<\sigma_a-1/2-\delta/4.
\]
From \eqref{allchoices} we see that $d\delta=2-d+2\rho_1$. Hence, $d-1-\rho_1+d\delta=1+\rho_1$, so 
$(d-1-\rho_1+d\delta)(\sigma_a-1/2+\varepsilon)-\delta/2)
<\sigma_a-1/2-\delta/4$ as above.

By choosing $p$ sufficiently large, we can ensure the last error term in \eqref{nexthkeqn} above is $O(T^{-\gamma})$,
for any $\gamma>0$, and that it is thus dominated by the two other terms.
We may summarise our arguments in this section as follows.
\begin{proposition}\label{directhk} 
Suppose $1<d<2$.  Then for $X_1=T^{1+\rho_1}>4T$ and $\alpha\in [T^{\delta},T^{\delta}+1]$, 
with $\delta$, $\varepsilon$ and $\rho_1$ chosen as in \eqref{allchoices},  
\begin{equation}\label{firsthkeqn}
\hktz=\sqrt{2\pi}\sum_{T<n<4T}a_ne^{-2\pi i\alpha n}e^{(-n/T^{1+\rho_1})^p}+O(T^{\sigma_a-\frac{1}{2}-\frac{\delta}{4}}).
\end{equation}  
\end{proposition}

\section{A second method of estimating $\hktz$}
In this section we will complete the proof of Theorem \ref{mainthm} by
evaluating the function $\hktz$ in a second way, assuming that $1<d$ (note
the strict inequality which we also assumed in the latter part of the previous subsection). This will allow us
to complete the proof of Theorem \ref{mainthm}.

We first use the functional equation \eqref{fnaleqn} in \eqref{hkdefn} to obtain
\begin{equation}\label{applyfefirsttwo}
\hktz=
\frac{1}{\sqrt{\alpha}}\int_{K_T}
\omega Q^{-2it}\frac{\tilde{G}(1/2-it)}{G(1/2+it)}
\tilde{F}(1/2-it)e^{it\log\left(\frac{t}{2\pi e\alpha}\right)-i\frac{\pi}{4}}dt.
\end{equation}

Using \eqref{fourthgammaest} together with 
Lemma \ref{basiclemma} for $\tilde{F}(1/2-it)$ (see Remark \ref{lemremark}), gives (for $X_2>0$, $0<\eta_2<1-x+p-\sigma_a$ and $u=u_2=-p+\eta_2$)
\begin{flalign}\label{postfehkrefined}
\hktz =\frac{e^{iB}\omega}{\sqrt{\alpha}}
\sum_{n=1}^{\infty}\frac{\overline{a_n}}{\sqrt{n}}e^{-(n/X_2)^{p}}
(J_n^{(3)}+J_n^{(4)})
+\tilde{R}_1(\alpha,T)+\tilde{R}_2(\alpha,T),
\end{flalign}
where
\begin{equation}
J_n^{(3)}=\int_{K_T}g_2(t)e^{if_2(t)}dt\quad\text{and}\quad J_n^{(4)}=\int_{K_T}g_2(t)e^{if_2(t)}O(1/t)dt\nonumber
\end{equation}
and $\tilde{R}_i(\alpha,T)=\frac{e^{iB}\omega}{\sqrt{\alpha}}\int_{K_T}\tilde{r}_i(t)g_2(t)e^{if_2^0(t)}dt$, where $\tilde{r}_i(t)$, $i=1,2$
are the error terms appearing in Lemma \ref{basiclemma} applied to $\tilde{F}(1/2-it)$. 
Here we have set
\begin{flalign}
&f_2^0(t)=-t(d-1)\log(e^{-1}(2\pi CQ^2\alpha)^{1/(d-1)}t)
-A\log t+B-\frac{\pi}{4},\nonumber\\
 &f_2(t)=-t(d-1)\log(e^{-1}(2\pi CQ^2\alpha n^{-1})^{1/(d-1)}t)
-A\log t+B-\frac{\pi}{4}\nonumber\\
&=-t(d-1)\log(e^{-1}(q\alpha n^{-1})^{1/d-1}t)-A\log t+B-\frac{\pi}{4},\nonumber
\end{flalign}
$g_2(t)=1$, $K=K_T$, and $q=2\pi CQ^2$ as before.
We have
\[
 f_2^{\prime}(t)=-(d-1)\log\left((q\alpha n^{-1})^{1/d-1}t\right)-A/t,
 \]
 and
\[
f_2^{\prime\prime}(t)=-(d-1)/t+A/t^2,
\]
We estimate the integrals $J_n^{(3)}$ as we did the integrals $I_n$
(in fact, the integral $J_n^{(3)}$ resembles the integral $J_n^{(1)}$
more closely). As before, we see that $f_2^{\prime}(t)\sim \log\left((q\alpha n^{-1})^{1/d-1}(t+v)\right)$.
We see easily (as in the analysis of $J_n^{(1)}$) that if $f_2^{\prime}(y_n)=0$,
then $y_n\sim \alpha^d T^{d-1}$ when $y_n\in K_T$. Assume that
$\alpha\in [T^{\delta},T^{\delta}+1]$ with $\delta$ chosen as in \eqref{allchoices}.
We will choose $X_2=T^{1+\rho_2}$ with $\rho_2=\rho_1$. Thus, 
\eqref{allchoices} shows that $1+\rho_2<d-1+d\delta$.
Since we are only 
interested in the sum upto $X_2^{1+\varepsilon}$, if 
$\varepsilon$ is also given by \eqref{allchoices}, we see that we have the stronger
inequality
$(1+\rho_2)(1+\varepsilon)<d-1+d\delta$, and the case $f_2^{\prime}(y_n)=0$ will not occur in our analysis. 
Using Lemma \ref{largen}, we see 
(exactly as in the analysis of $I_n$ and $J_n^{(1)}$) that
$J_n^{(3)}=O(1)$. The integral $J_n^{(4)}$ can be estimated trivially to give $O(1)$.
It follows that
\begin{equation}\label{sumjnthreeest}
\frac{e^{iB}\omega}{\sqrt{\alpha}}
\sum_{n=1}^{\infty}\frac{\overline{a_n}}{\sqrt{n}}e^{-(n/X_2)^p}(J_n^{(3)}+J_n^{(4)})
=O(\alpha^{-\frac{1}{2}}X_2^{\sigma_a-\frac{1}{2}+\varepsilon}).
\end{equation}
The error term $\tilde{R}_1(\alpha,T)$ is easily seen to decay exponentially in $\alpha$ and $T$, as before
(and, in fact, will not occur if $F(z)$ is entire). The term $\tilde{R}_2(\alpha,T)$ is estimated
exactly as $R_2(\alpha,T)$ was. Once again, our analysis will yield 
\begin{equation}\label{rtwotildeest}
\tilde{R}_2(\alpha,T)=
O(\alpha^{-\frac{1}{2}}(\alpha^dT^dX_2^{-1})^{\sigma_a-\frac{1}{2}+\varepsilon})
+O((\alpha^dT^{d-1})^{\sigma_a-\frac{1}{2}+\varepsilon}T^{p-\varepsilon+\frac{1}{2}}X_2^{\varepsilon-p}).
\end{equation}
Combining \eqref{sumjnthreeest} and \eqref{rtwotildeest} above gives us
\begin{flalign}\label{beforeinterhktzest}
\hktz&=O(\alpha^{-\frac{1}{2}}X_2^{\sigma_a-\frac{1}{2}+\varepsilon})
+O(\alpha^{-\frac{1}{2}}(\alpha^dT^dX_2^{-1})^{\sigma_a-\frac{1}{2}+\varepsilon})\nonumber\\
&+O((\alpha^dT^{d-1})^{\sigma_a-\frac{1}{2}+\varepsilon}T^{p-\varepsilon+\frac{1}{2}}X_2^{\varepsilon-p}).
\end{flalign}
Now assume that $X_2=T^{1+\rho_2}$.
Substituting for $X_2$ and $\alpha$ in the other two terms gives the estimate
\begin{equation}\label{interhktzest}
\hktz=O(T^{(1+\rho_2)(\sigma_a-\frac{1}{2}+\varepsilon)-\frac{\delta}{2}})
+O(T^{(d-1-\rho_2+d\delta)(\sigma_a-\frac{1}{2}+\varepsilon)-\frac{\delta}{2}}).
\end{equation}
Since we have chosen $\rho_2=\rho_1$, and $\delta$ and $\varepsilon$ as in \eqref{allchoices}, 
the exponents of $T$ in occurring in the equation above are the same as those occurring in the first two terms of \eqref{nexthkeqn}. 
If $p$ is chosen large enough, the third term in the right-hand side above 
is bounded by $O(T^{-\gamma})$ for any $\gamma>0$, and it is thus dominated by the other terms. We thus have the following proposition.
\begin{proposition}\label{postfehk} Suppose $1<d<2$. For $X_2=T^{1+\rho_2}$ and for
$\alpha\in [T^{\delta},T^{\delta}+1]$, where $\delta$,
$\varepsilon$ and  $\rho_2=\rho_1$ are chosen as in\eqref{allchoices}, we have
\begin{equation}\label{secondhkeqn}
\hktz=O(T^{\sigma_a-\frac{1}{2}-\frac{\delta}{4}})
\end{equation}
for some $\varepsilon_2>0$.
\end{proposition}

The proof of Theorem \ref{mainthm} follows almost immediately from Propositions
\ref{directhk} and \ref{postfehk}.
Combining equation \eqref{firsthkeqn} and equation \eqref{secondhkeqn}, we get
\[
\sqrt{2\pi}\sum_{T<n<4T}a_ne^{-2\pi i\alpha n}e^{-(n/T^{1+\rho_1})^p}=O(T^{\sigma_a-\frac{1}{2}-\frac{\delta}{4}})
\]
for $\alpha\in [T^{\delta},T^{\delta}+1]$.
Taking the $L^2$-norm yields (by Parseval's formula)
\[
\left(\sum_{T<n<4T} |a_n|^2\right)^{\frac{1}{2}}=O(T^{\sigma_a-\frac{1}{2}-\frac{\delta}{4}}).
\]
But this contradicts the $\Omega$-estimate \eqref{ltwolowerbound} which is implied by the fact that the 
abscissa of absolute convergence of $F(z)$ is $\sigma_a$. This proves Theorem \ref{mainthm}.

\section*{Acknowledgements} 
The authors are grateful to D. Surya Ramana for many helpful comments, suggestions and corrections. The second author would further like to acknowledge his constant encouragement and support for this project.

\bibliographystyle{alpha}
\bibliography{../../../../../Bibtex/master2020}

\end{document}